\documentclass[12pt,reqno]{amsart}
\usepackage{amscd,amsmath,amssymb,amsthm,amsfonts,epsfig,graphics}
\usepackage{geometry}
\usepackage{graphicx}
\usepackage{mathrsfs,amssymb}
\usepackage{bm}
\usepackage{url}
\usepackage{color}
\usepackage{amsmath}

\theoremstyle{plain}

\newtheorem{lemma}{Lemma}

\newtheorem{theorem}{Theorem}

\numberwithin{equation}{section}

\allowdisplaybreaks[4]
\begin{document}
\title[APPROXIMATING POINTWISE PRODUCTS OF QUASIMODES]
 {APPROXIMATING POINTWISE PRODUCTS OF QUASIMODES}

\author{Mei Ling Jin}

\address{Department of Mathematics, Harbin
Institute of Technology, Harbin $150001$, P.R. China}
\email{jinsmledu@gmail.com}

\address{}
\email{}



\dedicatory{}

\begin{abstract}
We obtain approximation bounds for products of quasimodes for the Laplace-Beltrami operator, on compact Riemannian manifolds of all dimensions without boundary. We approximate the products of quasimodes $uv$ by a low-degree vector space $B_{n}$, and we prove that the size of the space $\dim(B_{n})$ is small. In our paper, we first study bilinear quasimode estimates of all dimensions $d = 2, 3$, $d = 4,5$ and $d \ge 6$, respectively, to make the highest frequency disappear from the right hand. Furthermore, the result of the case $\lambda=\mu$ of bilinear quasimode estimates improves $L^{4}$ quasimodes estimates of Sogge-Zelditch in \cite{sogge6} when $d \ge 8$.  And on this basis, we give approximation bounds in $H^{-1}$ norm. We also prove approximation bounds for the products of  quasimodes in $L^{2}$ norm using the results of $L^{p}$-estimates for quasimodes in \cite{sogge3}. We extend the results of Lu-Steinerberger in \cite{lu} to quasimodes.
\end{abstract}

\maketitle

\section{Introduction}
Let $(M,g)$ be a compact Riemannian manifolds of dimension $d$ without boundary. We denote $\Delta = \Delta_{g}$ the Laplace-Beltrami operator associated to the metric $g$ on $M$. A quasimodes $u$ of $\Delta_{g}$ satisfies
\begin{align}\label{delta}
||(\Delta_{g} + \lambda^{2})u||_{2} \leq C_{0}\lambda&& \rm and &&||u||_{2} = 1,
\end{align}
where $C_{0}$ is independent of $\lambda$. Let ${u_{j}}$  be quasimodes with frequencies $\lambda_{j}$ arranged so that $0 = \lambda_{0} < \lambda_{1} \leq \lambda_{2} \leq \cdots.$ 

In our paper, we consider the problem about the function $u_{i}(x)u_{j}(x)$, for example,the size of $<u_{i}u_{j}, e_{k}>$. We extend the results in \cite{lu} to quasimodes. While there are many bilinear estimates for eigenfunctions and spectral clusters (see e.g. \cite{takaoka}, \cite{jjc}, \cite{hs}), there are very few results of bilinear estimates for quasimodes. Some results have been obtained in the presence of additional assumptions. For example, bilinear quasimodes estimates were studied for all dimensions under the assumption of microlocal support in \cite{guo1}.

Let $u$ be quasimodes satisfying \eqref{delta} with frequencies $\lambda$ and let $v$ be quasimodes satisfying \eqref{delta} with frequencies $\mu$, respectively. In the following, we want to approximate $uv$ by another vector space $B_{n}$ with lower degree $\dim(B_{n}) \sim c \cdot n$, often referred as density fitting in quantum chemistry literature.

To prove such a result, a first step is to prove the following bilinear estimates for quasimodes:
\begin{equation}\label{uv}
||uv||_{2} \lesssim \Lambda(d, \min(\lambda, \mu))(\lambda^{-1}||(-\Delta_{g}-\lambda^{2})u||_{2} + ||u||_{2})(\mu^{-1}||(-\Delta_{g}-\mu^{2})v||_{2} + ||v||_{2}),
\end{equation}
if $2 \leq d \leq 5$, 
and 
\begin{equation}\label{uv2}
\begin{aligned}
||uv||_{2} \lesssim &\Lambda(d,  \mu)(\lambda^{-1}||(-\Delta_{g}-\lambda^{2})u||_{2} + ||u||_{2})\\ &(\mu^{-1}||(-\Delta_{g}-\mu^{2})v||_{2} + ||v||_{2} + \mu^{-N+d/2-\sigma(q)}||(I-\Delta_{g})^{N/2}R_{\mu}v||_{2})
\end{aligned}
\end{equation}
if $d \ge 6$,
assuming that $\mu \leq \lambda$ and $N > d/2$ with $R_{\lambda}$ being the projection operator for $P = \sqrt{-\Delta_{g}}$ corresponding to the internal $[2\lambda, \infty)$. Furthermore, for $d = 6, 7$, we can get a better result than the result of \eqref{uv2} if $\lambda = \mu$ as follows (see more details in the proof of Theorem \ref{be4}):
\begin{equation}\label{qmr05}
||u||_{4}^{2} \lesssim \Lambda(d, \lambda)(\lambda^{-1}||(-\Delta_{g}-\lambda^{2})u||_{2} + ||u||_{2})^{2},
\end{equation}

where 
\begin{equation}\label{lanu}
\Lambda(d,\nu) := \begin{cases}
\nu^{\frac14} & \text{if } d=2,\\
\nu^{\frac12}\log^{\frac12}(\nu) & \text{if } d=3,\\
\nu^{\frac{d-2}{2}} & \text{if }d \ge 4.
\end{cases}
\end{equation}.

Bilinear quasimode estimates for all dimensions were obtained by Guo-Han-Tacy (\cite{guo1}). They assumed microlocal support in \cite{guo1}. The assumption is needed since the $L^{p}$ estimates of quasimodes do not hold for $p=\infty$ if $d=4$ or $p > \frac{2d}{d-4}$ if $d \ge 5$.

Notice that the case $\lambda = \mu$ when $d = 2, 3$ in (\ref{uv}) is a particular case of general $L^{p}$ estimates due to Sogge-Zelditch (\cite{sogge6}) which takes the following form, for every $p \in [2, +\infty]$,
\begin{equation}\label{u}
||u||_{L^{p}(M)} \leq C\lambda^{\sigma(p)} (\lambda^{-1}||(-\Delta_{g}-\lambda^{2})u||_{2} + ||u||_{2}),
\end{equation}
where $\sigma(4) = 1/8$ if $d = 2$, and more generally $\sigma(p) = \max(d(\frac{1}{2}-\frac{1}{p})-\frac{1}{2}, \frac{d-1}{2}(\frac{1}{2}-\frac{1}{p}))$.

And (\ref{u}) with semiclassical variants for all dimensions and all exponents $p \in (2, +\infty]$ was also given in \cite{ktz} under the assumption that $u$ is spectrally localized. This assumption is needed since (\ref{u}) does not hold for $p=\infty$ if $d=4$ or $p > \frac{2d}{d-4}$ if $d \ge 5$, as was explained in \cite{sogge3}.

The case $\lambda = \mu$ when $4 \leq d \leq 7$ in (\ref{uv}) is a particular case of $L^{4}$ estimates due to Sogge-Zelditch (\cite{sogge6}) which takes the following form, 
\begin{equation}\label{u1}
||u||_{L^{4}(M)} \leq C\lambda^{\sigma(4)} (\lambda^{-1}||(-\Delta_{g}-\lambda^{2})u||_{2} + ||u||_{2}),
\end{equation}
where $\sigma(4) = \frac{d-2}{4}$ if $4 \leq d \leq 7$.

Moreover, the result of the case $\lambda=\mu$ when $d \ge 8$ in (\ref{uv2}) improves $L^{4}$ quasimodes estimates of Sogge-Zelditch (\cite{sogge6}). 

Indeed, when $\lambda=\mu$ for $d \ge 8$ in (\ref{uv2}), we have
\begin{equation}\label{qmr00}
\begin{aligned}
||u||_{4}^{2}  &\lesssim \Lambda(d, \lambda)(\lambda^{-1}||(-\Delta_{g}-\lambda^{2})u||_{2} + ||u||_{2})\\
&(\lambda^{-1}||(-\Delta_{g}-\lambda^{2})u||_{2} + ||u||_{2} + \lambda^{-N+d/2-\sigma(4)}||(I-\Delta)^{N/2}R_{\lambda}u||_{2})\\&=\lambda^{2\sigma(4)}(\lambda^{-1}||(-\Delta_{g}-\lambda^{2})u||_{2} + ||u||_{2})\\
&(\lambda^{-1}||(-\Delta_{g}-\lambda^{2})u||_{2} + ||u||_{2} + \lambda^{-N+d/2-\sigma(4)}||(I-\Delta)^{N/2}R_{\lambda}u||_{2}).
\end{aligned}
\end{equation}
And according to the result in \cite{sogge6}, we have
\begin{equation}\label{qmr01}
\begin{aligned}
||u||_{4}^{2}  &\lesssim \lambda^{2\sigma(4)}
(\lambda^{-1}||(-\Delta_{g}-\lambda^{2})u||_{2} + ||u||_{2} + \lambda^{-N+d/2-\sigma(4)}||(I-\Delta)^{N/2}R_{\lambda}u||_{2})^{2}.
\end{aligned}
\end{equation}

The proof of bilinear estimates for quasimodes is based on the following results of bilinear estimates for spectral clusters in \cite{burq1} of dimension 2, in \cite{burq2} of dimension 3 and in \cite{burq3, burq4} of higher dimensions (see the similar techniques in \cite{sogge6}):
\begin{equation}\label{chiuv}
||\chi_{\lambda}u\chi_{\mu}v||_{2} \lesssim \Lambda(d, min(\lambda, \mu))||u||_{2}||v||_{2},   
\end{equation}
where $\chi_{\lambda}$ is the projection operators corresponding to the unit internal $[\lambda, \lambda + 1]\}$.

Notice that the case $\lambda = \mu$ in (\ref{chiuv}) is a particular case of general $L^{p}$ estimates due to Sogge (\cite{sogge2}, \cite{sogge4}, \cite{sogge5}) which take the following form, for every $p \in [2, +\infty]$,
\begin{equation}\label{chiu}
||\chi_{\lambda}u||_{L^{p}(M)} \leq C\lambda^{\sigma(p)}||u||_{L^{2}(M)}, 
\end{equation}
where $\sigma(4) = 1/8$ if $d = 2$.

The advantages of using estimates (\ref{uv}), rather than applying Holder's inequality and (\ref{u}), is to make the highest frequency disappear from the right hand side, which is crucial in the nonlinear analysis. And we can also get a better results of the approximation pointwise products of quasimodes.

Then the next step is to use bilinear estimates for quasimodes of Theorem \ref{be2} and Theorem \ref{be3} to prove Theorem \ref{fix} and Theorem \ref{h1}.

The rest of this paper is organized as follows. In Section 2 we give the main result, $H^{-1}$ approximation. Section 3 consists in two parts. First we prove bilinear estimates for quasimodes on a two dimensional compact Riemannian manifold without boundary. Then we obtain bilinear quasimode estimates on a three dimensinal compact Riemannian manifold. In Section 4 the bilinear quasimode estimates for higher dimensions are also obtained. Section 5 is denoted to the proof of theorem \ref{h1} based on the results of Theorem \ref{be2}, Theorem \ref{be3} and Theorem \ref{be4}. In section 6, we show that the product of quasimodes $uv$ can be well approximated by the elements of $B_{\nu}$ if $\nu$ is not much larger than $n$, based on the estimates of quasimodes in \cite{ktz}. We improve the results in Lu-Steinerberger (\cite{lu}), since we can get approximation of products for quasimodes. Section 7 is denoted to the proof of Thereom \ref{fix} and Lemma \ref{for}.

\section{Main results}
Motivated by the above, for $1 \leq i \leq j \leq n$, we are interested in estimating products of quasimodes $u_{i}u_{j}$ by finite linear combinations of quasimodes. With this in mind, let for $\nu \in \mathbb{N}$
\begin{equation}\label{E}
E_{\nu}f = \sum_{k=0}^{\nu}< f, e_{k}> e_{k}
\end{equation}
denote the projection of $f \in L^{2}(M)$ onto the space, $B_{\nu}$, spanned by $\{e_{k}\}_{k=0}^{\nu}$, and 
\begin{equation}\label{R}
\mathcal{R}_{\nu}f = f - E_{\nu}f
\end{equation}
denote the ``remainder term" for this projection. Thus,
\begin{equation}\label{f}
||f - E_{\nu}f||_{L^{2}(M)}^{2} = ||\mathcal{R}_{\nu}f||_{2}^{2} = \sum_{k>\nu}|<f,e_{k}>|^{2}.
\end{equation}
Since a multiple of $|x-y|^{-1}$ is the fundamental solution of the Laplacian in $\mathbb{R}^{3}$, the appropriate physically relevant problems involve the Sobolev space $H^{-1}$ equipped with the norm defined by
\begin{equation}\label{}
||f||_{H^{-1}}^{2} = ||(I - \Delta_{g})^{-1/2}f||_{L_{2}}^{2} = \sum(1 + \lambda_{k}^{2})^{-1}|\left<f, e_{k}\right>|^{2}.
\end{equation}

Our $H^{-1}$ approximation result then is the following.
\begin{theorem}\label{h1}
Let $(M,g)$ be a compact Riemannian manifold of dimension $d$. Then if $1 \leq i \leq n < \nu$, $1 \leq j \leq m < \nu$ and $\varepsilon \in (0,1)$ we have 
\begin{equation}\label{2.2}
||\mathcal{R}_{\nu}(u_{i}u_{j})||_{H^{-1}}<\varepsilon,    \ \text{if} \ \nu = O(\Omega(d, \min(\lambda_{m}, \lambda_{n}))\varepsilon^{-d}),
\end{equation}
\end{theorem}
where 
\begin{equation}\label{2.3}
\Omega(d, \mu) = 
             \begin{cases} 
\mu^{\frac{1}{2}} & \text{if } d = 2,\\
                    \mu^{\frac{3}{2}}\log^{\frac{3}{2}}(\mu) & \text{if } d=3,\\
                    \mu^{\frac{d(d-2)}{2}} &  \text{if } d \ge 4.
                 \end{cases}  
\end{equation}

\section{Bilinear quasimode estimates in two and three dimensions}

We introduce the following notation: given $\mu \ge 1$, we set
\begin{equation}\label{nu}
\Lambda(d,\mu) :=  
             \begin{cases} 
                 \mu^{\frac14} &  \text{if }  d = 2,\\
                   \mu^{\frac12}\log^{\frac12}(\mu) & \text{if } d=3.
                 \end{cases}  
\end{equation}
\begin{theorem}\label{be2}
$(M,g)$ is a 2 dimensional compact Riemannian manifold without boundary. Let $u$ be quasimodes satisfying (\ref{delta}) with frequencies $\lambda$ and let $v$ be quasimodes satisfying (\ref{delta}) with frequencies $\mu$, respectively. Assume that $1 \leq \mu \leq \lambda$. And $\{\lambda_{j}^{2}\}$ are the eigenvalues of $- \Delta_{g}$. Then the following bilinear estimates hold
\begin{equation}\label{qmr3} 
||uv||_{2} \lesssim \Lambda(d, \mu)(\lambda^{-1}||(-\Delta_{g}-\lambda^{2})u||_{2} + ||u||_{2})(\mu^{-1}||(-\Delta_{g}-\mu^{2})v||_{2} + ||v||_{2}).
\end{equation}
\end{theorem}
\begin{proof}
We split our $u$ satisfying (\ref{delta}) into ``low" and ``high" frequency parts. To this end we choose $0 \leq \psi \in C^{\infty}(\mathbb{R})$ satisfying $\psi(r) = 1$ for $r \leq 2$ and $\psi(r) = 0$ for $r \ge 4$. We then set $\rho(r) = 1 - \psi(r)$ and note that $\rho(r) = 0$ if $r < 2$. \\
We can now write 
\begin{equation}
 u = L_{\lambda}u + H_{\lambda}u,   
\end{equation}
where $L_{\lambda} = \psi(P/\lambda)$ and $H_{\lambda} = \rho(P/\lambda)$. Here $P = \sqrt{-\Delta_{g}}$.\\
Since
\begin{equation}
||uv||_{2} \leq ||vL_{\lambda}u||_{2} + ||vH_{\lambda}u||_{2},  
\end{equation}
it is enough to prove that 
\begin{equation}
||vL_{\lambda}u||_{2} \lesssim \Lambda(d, \mu)(\lambda^{-1}||(-\Delta_{g}-\lambda^{2})u||_{2} + ||u||_{2})(\mu^{-1}||(-\Delta_{g}-\mu^{2})v||_{2} + ||v||_{2}), 
\end{equation}
\\and
\begin{equation}
||vH_{\lambda}u||_{2} \lesssim \Lambda(d, \mu)(\lambda^{-1}||(-\Delta_{g}-\lambda^{2})u||_{2} + ||u||_{2})(\mu^{-1}||(-\Delta_{g}-\mu^{2})v||_{2} + ||v||_{2}). 
\end{equation}
To estimate the high frequency part, the assumptions on $p=4$ then ensures that
\begin{equation}
 (1 - \Delta_{g})^{-(\frac{1}{2}-\frac{1}{p})} : L^{2} \rightarrow L^{p} \ \ \and \ \ \frac{1}{2}-\frac{1}{p} \leq 1.  
\end{equation}
Thus 
\begin{equation}\label{Hu}
\begin{aligned}
 ||H_{\lambda}u||_{4} &\lesssim ||P^{1/2}H_{\lambda}u||_{2}\\
 &\lesssim \lambda^{-3/2}||(-\Delta_{g}-\lambda^{2})H_{\lambda}u||_{2}\\
 &\leq \lambda^{-3/2}||(-\Delta_{g}-\lambda^{2})u||_{2} \leq \lambda^{-1/2}(\lambda^{-1}||(-\Delta_{g}-\lambda^{2})u||_{2} + ||u||_{2}),
 \end{aligned}
\end{equation}
using in the second inequality the fact that $H_{\lambda}u$ only has frequencies larger than $2\lambda$, as well as $2(\frac{1}{2}-\frac{1}{p}) \leq 2$, due to our assumptions on $p$.\\
If we combine (\ref{Hu}) and the $L^{p}$ estimates of quasimodes in \cite{sogge3}, we have
\begin{equation}
\begin{aligned}
||vH_{\lambda}u||_{2} &\leq ||H_{\lambda}u||_{4}||v||_{4}\\ 
&\leq \lambda^{-1/2}(\lambda^{-1}||(-\Delta_{g}-\lambda^{2})u||_{2} + ||u||_{2})\mu^{1/8}(\mu^{-1}||(-\Delta_{g}-\mu^{2})v||_{2} + ||v||_{2})\\
&<\lambda^{-3/8}(\lambda^{-1}||(-\Delta_{g}-\lambda^{2})u||_{2} + ||u||_{2})(\mu^{-1}||(-\Delta_{g}-\mu^{2})v||_{2} + ||v||_{2}).
\end{aligned}
\end{equation}
When $p=\infty$, we use the fact that 
\begin{equation*}
(1 - \Delta_{g})^{-\frac{1}{2}-\varepsilon}  : L^{2} \rightarrow L^{\infty}   
\end{equation*}
if $\varepsilon >0$, as well as the fact that $\frac{1}{2} + \varepsilon < 1$ if $\varepsilon < \frac{1}{2}$. Based on this, we have
\begin{equation}
 \begin{aligned}
||H_{\mu}v||_{\infty} &\lesssim ||P^{1+\varepsilon}H_{\mu}v||_{2}\\
&\leq \mu^{\varepsilon} \mu^{-1}||(-\Delta_{g}-\mu^{2})H_{\mu}v||_{2}\\
&\leq \mu^{\varepsilon} \mu^{-1}||(-\Delta_{g}-\mu^{2})v||_{2}\\
&\leq \mu^{1/4}(\mu^{-1}||(-\Delta_{g}-\mu^{2})v||_{2} + ||v||_{2})
 \end{aligned}   
\end{equation}
Thus,
\begin{equation}
\begin{aligned}
 ||L_{\lambda}uH_{\mu}v||_{2} &\leq ||L_{\lambda}u||_{2}||H_{\mu}v||_{\infty}\\ 
 &\lesssim ||u||_{2} \mu^{1/4}(\mu^{-1}||(-\Delta_{g}-\mu^{2})v||_{2} + ||v||_{2})\\
 &\leq \mu^{1/4}(\lambda^{-1}||(-\Delta_{g}-\lambda^{2})u||_{2} + ||u||_{2})(\mu^{-1}||(-\Delta_{g}-\mu^{2})v||_{2} + ||v||_{2}).
\end{aligned}
\end{equation}
To prove the remaining part of estimates, we write
\begin{equation}
L_{\lambda}u = \sum_{k \leq 2\lambda}\chi_{k}u, \ \ \  L_{\mu}v = \sum_{j \leq 2\mu} \chi_{j}v.    
\end{equation}
where $\chi_{k}$ is the projection operators corresponding to the unit internal $[k, k + 1]$ and $\chi_{j}$ is the projection operators corresponding to the unit internal $[j, j + 1]$.

By Cauchy-Schwarz inequality, we have
\begin{equation}
|L_{\lambda}uL_{\mu}v| \leq (\sum_{k \leq 2\lambda, j \leq 2\mu}|(1+|k-\lambda|)\chi_{k}u(1+|j-\mu|)\chi_{j}v|^{2})^{1/2}    
\end{equation}
Thus,
\begin{equation}
\begin{aligned}
||L_{\lambda}uL_{\mu}v||_{2}^{2} &\leq \sum_{k \leq 2\lambda, j \leq 2\mu}(1+|k-\lambda|)^{2}(1+|j-\mu|)^{2}||\chi_{k}u\chi_{j}v||_{2}^{2}\\
&\leq \sum_{k \leq 2\lambda, j \leq 2\mu} (\min(j,k))^{1/2}||(1+|\lambda-k|\chi_{k}u)||_{2}^{2}||(1+|\mu-j|)\chi_{j}v||_{2}^{2}\\
&\leq \mu^{1/2}(\sum_{n=0}^{\infty}||(1+|k-\lambda|)\chi_{k}u||_{2}^{2})(\sum_{n=0}^{\infty}||(1+|j-\mu|)\chi_{j}v||_{2}^{2})\\
& \leq \mu^{1/2}(\lambda^{-1}||(-\Delta_{g}-\lambda^{2})u||_{2} + ||u||_{2})^{2}(\mu^{-1}||(-\Delta_{g}-\mu^{2})v||_{2} + ||v||_{2})^{2},
\end{aligned}
\end{equation}
since the estimates for spectral clusters in \cite{burq1}:
\begin{equation*}
||\chi_{\lambda}u\chi_{\mu}v||_{2} \lesssim (\min(\lambda, \mu))^{1/4}||u||_{2}||v||_{2}.    
\end{equation*}
This completes the proof of theorem \ref{be2}.
\end{proof}

\begin{theorem}\label{be3}
$(M,g)$ is a 3 dimensional compact Riemannian manifold without boundary. Let $u$ be quasimodes satisfying (\ref{delta}) with frequencies $\lambda$ and let $v$ be quasimodes satisfying (\ref{delta}) with frequencies $\mu$, respectively. Assume that $1 \leq \mu \leq \lambda$. Then the following bilinear estimates holds
\begin{equation}\label{qmr31}
||uv||_{2} \lesssim \Lambda(d, \mu)(\lambda^{-1}||(-\Delta_{g}-\lambda^{2})u||_{2} + ||u||_{2})(\mu^{-1}||(-\Delta_{g}-\mu^{2})v||_{2} + ||v||_{2}).
\end{equation}
\end{theorem}
\begin{proof}
As before,
\begin{equation}
||uv||_{2} \leq ||vL_{\lambda}u||_{2} + ||vH_{\lambda}u||_{2}.  
\end{equation}
To estimate the high frequency part, the assumptions on $p=4$ then ensure that
\begin{equation}
 (1 - \Delta_{g})^{-\frac{3}{2}(\frac{1}{2}-\frac{1}{p})} : L^{2} \rightarrow L^{p} \ \ \ \and \ \ \  \frac{3}{2}(\frac{1}{2}-\frac{1}{p}) \leq 1.
\end{equation}
Thus 
\begin{equation}\label{Hu3}
\begin{aligned}
 ||H_{\lambda}u||_{4} &\lesssim ||P^{3/4}H_{\lambda}u||_{2}\\
 &\lesssim \lambda^{-5/4}||(-\Delta_{g}-\lambda^{2})H_{\lambda}u||_{2}\\
 &\leq \lambda^{-5/4}||(-\Delta_{g}-\lambda^{2})u||_{2} \leq \lambda^{-1/4}(\lambda^{-1}||(-\Delta_{g}-\lambda^{2})u||_{2} + ||u||_{2}),
 \end{aligned}
\end{equation}
using in the second inequality the fact that $H_{\lambda}u$ only has frequencies larger than $2\lambda$, as well as $3(\frac{1}{2}-\frac{1}{p}) \leq 2$, due to our assumptions on $p$.\\
If we combine (\ref{Hu3}) and the $L^{p}$ estimates of quasimodes in \cite{sogge3}, we have
\begin{equation}
\begin{aligned}
||vH_{\lambda}u||_{2} &\leq ||H_{\lambda}u||_{4}||v||_{4}\\ &\leq \lambda^{-1/4}(\lambda^{-1}||(-\Delta_{g}-\lambda^{2})u||_{2} + ||u||_{2})\mu^{1/4}(\mu^{-1}||(-\Delta_{g}-\mu^{2})v||_{2} + ||v||_{2})\\
&<\mu^{1/2}(\lambda^{-1}||(-\Delta_{g}-\lambda^{2})u||_{2} + ||u||_{2})(\mu^{-1}||(-\Delta_{g}-\mu^{2})v||_{2} + ||v||_{2})\\
&<\Lambda(d, \mu)(\lambda^{-1}||(-\Delta_{g}-\lambda^{2})u||_{2} + ||u||_{2})(\mu^{-1}||(-\Delta_{g}-\mu^{2})v||_{2} + ||v||_{2}).
\end{aligned}
\end{equation}
So we need to estimate $||vL_{\lambda}u||_{2}$ next.
Since
\begin{equation*}
||vL_{\lambda}u||_{2} \leq ||L_{\lambda}uL_{\mu}v||_{2} + ||L_{\lambda}uH_{\mu}v||_{2},
\end{equation*}
it is enough to estimate 
\begin{equation}
||L_{\lambda}uL_{\mu}v||_{2} \leq \Lambda(d, \mu)(\lambda^{-1}||(-\Delta_{g}-\lambda^{2})u||_{2} + ||u||_{2})(\mu^{-1}||(-\Delta_{g}-\mu^{2})v||_{2} + ||v||_{2}),  
\end{equation}
and
\begin{equation}
||L_{\lambda}uH_{\mu}v||_{2} \leq \Lambda(d, \mu)(\lambda^{-1}||(-\Delta_{g}-\lambda^{2})u||_{2} + ||u||_{2})(\mu^{-1}||(-\Delta_{g}-\mu^{2})v||_{2} + ||v||_{2}).   
\end{equation}
To estimate the last term, we fix a nonnegative Littlewood-Paley bump function $\beta \in C_{0}^{\infty}((1/2, 2))$ satisfying $1 = \sum_{k=-\infty}^{\infty}\beta(r/2^{k}), r>0.$ And we denote by $S_{j}f = \beta(2^{-j}P)f$, then we have
\begin{equation}
\begin{aligned}
||H_{\mu}v||_{\infty} &\leq \sum_{2^{j}\ge2\mu}||S_{j}v||_{\infty}\\
&\leq \sum_{2^{j}\ge2\mu}2^{3j/2}||S_{j}v||_{2}\\
&\leq \sum_{2^{j}\ge2\mu}2^{-j/2}||-\Delta_{g}(S_{j}v)||_{2}\\
&\leq \mu^{-1/2}\sum_{2^{j}\ge2\mu}(\mu2^{-j})^{1/2}||-\Delta_{g}(S_{j}v)||_{2}\\
&\leq \mu^{-1/2}\sum_{2^{j}\ge2\mu}(\mu2^{-j})^{1/2}||-\Delta_{g}H_{\mu}v||_{2}\\
&\leq C\mu^{-1/2}||(-\Delta_{g}-\mu^{2})v||_{2}\\
&= C\mu^{1/2}\mu^{-1}||(-\Delta_{g}-\mu^{2})v||_{2}\\
&\leq C\mu^{1/2}(\mu^{-1}||(-\Delta_{g}-\mu^{2})v||_{2} + ||v||_{2})\\
&\lesssim \Lambda(d, \mu)(\mu^{-1}||(-\Delta_{g}-\mu^{2})v||_{2} + ||v||_{2})\\
\end{aligned}
\end{equation}
Thus,
\begin{equation}
\begin{aligned}
||L_{\lambda}uH_{\mu}v||_{2} &\leq ||L_{\lambda}u||_{2}||H_{\mu}v||_{\infty}\\  
&\lesssim ||u||_{2}\mu^{1/2}(\mu^{-1}||(-\Delta_{g}-\mu^{2})v||_{2} + ||v||_{2})\\
&\leq \mu^{1/2}(\lambda^{-1}||(-\Delta_{g}-\lambda^{2})u||_{2} + ||u||_{2})(\mu^{-1}||(-\Delta_{g}-\mu^{2})v||_{2} + ||v||_{2})\\
&\leq \Lambda(d, \mu)(\lambda^{-1}||(-\Delta_{g}-\lambda^{2})u||_{2} + ||u||_{2})(\mu^{-1}||(-\Delta_{g}-\mu^{2})v||_{2} + ||v||_{2}).
\end{aligned}
\end{equation}
To estimate the remaining part $||L_{\lambda}uL_{\mu}v||_{2}$, we write
\begin{equation}
L_{\lambda}u = \sum_{k \leq 2\lambda}\chi_{k}u, \ \ \ \  L_{\mu}v = \sum_{j \leq 2\mu} \chi_{j}v,   
\end{equation}
where $\chi_{k}$ is the projection operators corresponding to the unit internal $[k, k + 1]$ and $\chi_{j}$ is the projection operators corresponding to the unit internal $[j, j + 1]$.\\
By Cauchy-Schwarz inequality, we have
\begin{equation}
|L_{\lambda}uL_{\mu}v| \leq (\sum_{k \leq 2\lambda, j \leq 2\mu}|(1+|k-\lambda|)\chi_{k}u(1+|j-\mu|)\chi_{j}v|^{2})^{1/2}  
\end{equation}
Thus,
\begin{equation}
\begin{aligned}
||L_{\lambda}u&L_{\mu}v||_{2}^{2} \leq \sum_{k \leq 2\lambda, j \leq 2\mu}(1+|k-\lambda|)^{2}(1+|j-\mu|)^{2}||\chi_{k}u\chi_{j}v||_{2}^{2}\\
&\leq \sum_{k \leq 2\lambda, j \leq 2\mu} (\Lambda(d, \mu))^{2}||(1+|\lambda-k|\chi_{k}u)||_{2}^{2}||(1+|\mu-j|)\chi_{j}v||_{2}^{2}\\
&\leq (\Lambda(d, \mu))^{2}(\sum_{n=0}^{\infty}||(1+|k-\lambda|)\chi_{k}u||_{2}^{2})(\sum_{n=0}^{\infty}||(1+|j-\mu|)\chi_{j}v||_{2}^{2})\\
& \leq (\Lambda(d, \mu))^{2}(\lambda^{-1}||(-\Delta_{g}-\lambda^{2})u||_{2} + ||u||_{2})(\mu^{-1}||(-\Delta_{g}-\mu^{2})v||_{2} + ||v||_{2}),
\end{aligned}
\end{equation}
since the estimates for spectral clusters in \cite{burq2} of dimension 3:
\begin{equation*}
||\chi_{k}u\chi_{j}v||_{2} \lesssim \Lambda(d, \min(k, j))||u||_{2}||v||_{2}.    
\end{equation*}
This completes the proof of theorem \ref{be3}. 
\end{proof}

\section{Bilinear quasimodes estimates for higher dimensions}\label{higher}

We introduce the following notation: 
given $\nu \ge 1$, we set
\begin{equation}\label{nu1}
\Lambda(d,\mu) := \mu^{{\frac{d-2}{2}}}, \ \  \text{if } d\ge 4. 
\end{equation}

\begin{theorem}\label{be4}
$(M,g)$ is a $d$ ($d = 4, 5$) dimensional compact Riemannian manifold without boundary. Let $u$ be quasimodes satisfying (\ref{delta}) with frequencies $\lambda$ and let $v$ be quasimodes satisfying (\ref{delta}) with frequencies $\mu$, respectively. Assume that $1 \leq \mu \leq \lambda$. Then the following bilinear estimates holds
\begin{equation}\label{qmr4}
||uv||_{2} \lesssim \Lambda(d, \mu)(\lambda^{-1}||(-\Delta_{g}-\lambda^{2})u||_{2} + ||u||_{2})(\mu^{-1}||(-\Delta_{g}-\mu^{2})v||_{2} + ||v||_{2}).
\end{equation}
If $d \ge 6$ we also have for such $d$
\begin{equation}\label{qmr5}
\begin{aligned}
||uv||_{2}  \lesssim &\Lambda(d, \mu)(\lambda^{-1}||(-\Delta_{g}-\lambda^{2})u||_{2} + ||u||_{2})\\
&(\mu^{-1}||(-\Delta_{g}-\mu^{2})v||_{2} + ||v||_{2} + \mu^{-N+d/2-\sigma(q)}||(I-\Delta)^{N/2}R_{\mu}v||_{2}).
\end{aligned}
\end{equation}
assuming that $N > d/2$ with $R_{\lambda}$ being the projection operator for $P = \sqrt{-\Delta_{g}}$ corresponding to the internal $[2\lambda, \infty)$. Furthermore, for $d = 6, 7$, we can get a better result than the result of (\ref{qmr5}) if $\lambda = \mu$ as follows:
\begin{equation}\label{qmr051}
||u||_{4}^{2} \lesssim \Lambda(d, \lambda)(\lambda^{-1}||(-\Delta_{g}-\lambda^{2})u||_{2} + ||u||_{2})^{2}.
\end{equation}
\end{theorem}
To prove this theorem, we need use a special case of Theorem 1.3 in \cite{sogge3} corresponding to $V \equiv 0$.

\begin{proof}
As before,
\begin{equation}
||uv||_{2} \leq ||vL_{\lambda}u||_{2} + ||vH_{\lambda}u||_{2},  
\end{equation}
it is enough to prove that for $d = 4, 5$
\begin{equation}
||vL_{\lambda}u||_{2} \lesssim \Lambda(d, \mu)(\lambda^{-1}||(-\Delta_{g}-\lambda^{2})u||_{2} + ||u||_{2})(\mu^{-1}||(-\Delta_{g}-\mu^{2})v||_{2} + ||v||_{2}), 
\end{equation}
\\and
\begin{equation}
||vH_{\lambda}u||_{2} \lesssim \Lambda(d, \mu)(\lambda^{-1}||(-\Delta_{g}-\lambda^{2})u||_{2} + ||u||_{2})(\mu^{-1}||(-\Delta_{g}-\mu^{2})v||_{2} + ||v||_{2}). 
\end{equation}
To estimate the high frequency part, the assumptions on $p \leq \frac{2d}{d-4}$ for $d$ then ensure that
\begin{equation}
 (1 - \Delta_{g})^{-\frac{d}{2}(\frac{1}{2}-\frac{1}{p})} : L^{2} \rightarrow L^{p} \ \ \and \ \ \frac{d}{2}(\frac{1}{2}-\frac{1}{p}) \leq 1.  
\end{equation}
Thus 
\begin{equation}\label{Hu10}
\begin{aligned}
 ||H_{\lambda}u||_{p} &\lesssim ||P^{d(\frac{1}{2}-\frac{1}{p})}H_{\lambda}u||_{2}\\
 &\lesssim \lambda^{d(\frac{1}{2}-\frac{1}{p})-2}||(-\Delta_{g}-\lambda^{2})H_{\lambda}u||_{2}\\
 &\leq \lambda^{d(\frac{1}{2}-\frac{1}{p})-2}||(-\Delta_{g}-\lambda^{2})u||_{2} \leq \lambda^{d(\frac{1}{2}-\frac{1}{p})-1}(\lambda^{-1}||(-\Delta_{g}-\lambda^{2})u||_{2} + ||u||_{2}).
 \end{aligned}
\end{equation}
If we combine (\ref{Hu10}) and the $L^{p}$ estimates of Theorem 1.3 in \cite{sogge3} corresponding to $V \equiv 0$ when $d = 4, 5$, we have
\begin{equation}\label{vhu4}
\begin{aligned}
||v&H_{\lambda}u||_{2} \leq ||H_{\lambda}u||_{p}||v||_{q}\\ 
&\leq \lambda^{d(\frac{1}{2}-\frac{1}{p})-1}(\lambda^{-1}||(-\Delta_{g}-\lambda^{2})u||_{2} + ||u||_{2})\mu^{\sigma(q)}(\mu^{-1}||(-\Delta_{g}-\mu^{2})v||_{2} + ||v||_{2})\\
&\leq \mu^{\frac{d-2}{2}}(\lambda^{-1}||(-\Delta_{g}-\lambda^{2})u||_{2} + ||u||_{2})(\mu^{-1}||(-\Delta_{g}-\mu^{2})v||_{2} + ||v||_{2}).
\end{aligned}
\end{equation}
Since we want to make the highest frequency disappear from the right hand in the last inequality of (\ref{vhu4}), we need ensure that the power of $\lambda$ is negative. We need assume that $d(\frac{1}{2}-\frac{1}{p})-1 \leq 0$, that is, $p \leq \frac{2d}{d-2}$. We have $q = \frac{2p}{p-2}$ due to Holder inequality in the first inequality of (\ref{vhu4}). Therefore, we have $q < \frac{2d}{d-4}$ when $d = 4, 5$. Then we can estimate $||v||_{q}$ in the second inequality of (\ref{vhu4}). We also use the fact that $\lambda^{d(\frac{1}{2}-\frac{1}{p})-1} \mu^{\sigma(q)} \leq \mu^{\frac{d-2}{2}}$ in the last inequality of (\ref{vhu4}). Indeed, when $q \leq \frac{2(d+1)}{d-1}$, $\sigma(q) = \frac{d-1}{2}(\frac{1}{2}-\frac{1}{q})$. Then we have
$d(\frac{1}{2}-\frac{1}{p})-1 + \sigma(q) = d(\frac{1}{2}-\frac{1}{p})-1 + \frac{d-1}{2}(\frac{1}{2}-\frac{1}{q}) \leq \frac{d-2}{2}.$ When $q \ge \frac{2(d+1)}{d-1}$, $\sigma(q) = d(\frac{1}{2}-\frac{1}{q}) - \frac{1}{2}$. Then we have
$d(\frac{1}{2}-\frac{1}{p})-1 + \sigma(q) = d(\frac{1}{2}-\frac{1}{p})-1 + d(\frac{1}{2}-\frac{1}{q}) - \frac{1}{2} = \frac{d-3}{2} < \frac{d-2}{2}.$

So we need to estimate $||vL_{\lambda}u||_{2}$ next.
And we have
\begin{equation}\label{vlu4}
||vL_{\lambda}u||_{2} \leq ||L_{\lambda}uL_{\mu}v||_{2} + ||L_{\lambda}uH_{\mu}v||_{2}.
\end{equation}

To estimate the last term in (\ref{vlu4}), we fix a nonnegative Littlewood-Paley bump function $\beta \in C_{0}^{\infty}((1/2, 2))$ satisfying $1 = \sum_{k=-\infty}^{\infty}\beta(r/2^{k}), r>0.$ And we denote by $S_{j}f = \beta(2^{-j}P)f$, then we have
\begin{equation}
\begin{aligned}
||H_{\mu}v||_{\infty} &\leq \sum_{2^{j}\ge2\mu}||S_{j}v||_{\infty}\\
&\leq \sum_{2^{j}\ge2\mu}2^{dj/2}||S_{j}v||_{2}\\
&\leq \sum_{2^{j}\ge2\mu}2^{(d-4)j/2}||-\Delta_{g}(S_{j}v)||_{2}\\
&\leq \mu^{(d-4)/2}\sum_{2^{j}\ge2\mu}(\mu2^{-j})^{(4-d)/2}||-\Delta_{g}(S_{j}v)||_{2}\\
&\leq \mu^{(d-4)/2}\sum_{2^{j}\ge2\mu}(\mu2^{-j})^{(4-d)/2}||-\Delta_{g}H_{\mu}v||_{2}\\
&\leq C\mu^{(d-4)/2}||(-\Delta_{g}-\mu^{2})v||_{2}\\
&= C\mu^{(d-2)/2}\mu^{-1}||(-\Delta_{g}-\mu^{2})v||_{2}\\
&\leq C\mu^{(d-2)/2}(\mu^{-1}||(-\Delta_{g}-\mu^{2})v||_{2} + ||v||_{2})\\
&\lesssim \Lambda(d, \mu)(\mu^{-1}||(-\Delta_{g}-\mu^{2})v||_{2} + ||v||_{2})\\
\end{aligned}
\end{equation}
Thus,
\begin{equation}\label{411}
\begin{aligned}
 ||L_{\lambda}uH_{\mu}v||_{2} &\leq ||L_{\lambda}u||_{2}||H_{\mu}v||_{\infty}\\ 
 &\lesssim ||u||_{2} \Lambda(d, \mu)(\mu^{-1}||(-\Delta_{g}-\mu^{2})v||_{2} + ||v||_{2})\\
 &\leq \Lambda(d, \mu)(\lambda^{-1}||(-\Delta_{g}-\lambda^{2})u||_{2} + ||u||_{2})(\mu^{-1}||(-\Delta_{g}-\mu^{2})v||_{2} + ||v||_{2}).
\end{aligned}
\end{equation}
To prove the remaining part of estimates, we write
\begin{equation}
L_{\lambda}u = \sum_{k \leq 2\lambda}\chi_{k}u, \ \ \  L_{\mu}v = \sum_{j \leq 2\mu} \chi_{j}v,    
\end{equation}
where $\chi_{k}$ is the projection operators corresponding to the unit internal $[k, k + 1]$ and $\chi_{j}$ is the projection operators corresponding to the unit internal $[j, j + 1]$.

By Cauchy-Schwarz inequality, we have
\begin{equation}
|L_{\lambda}uL_{\mu}v| \leq (\sum_{k \leq 2\lambda, j \leq 2\mu}|(1+|k-\lambda|)\chi_{k}u(1+|j-\mu|)\chi_{j}v|^{2})^{1/2}    
\end{equation}
Thus,
\begin{equation}\label{414}
\begin{aligned}
||L_{\lambda}u&L_{\mu}v||_{2}^{2} \leq \sum_{k \leq 2\lambda, j \leq 2\mu}(1+|k-\lambda|)^{2}(1+|j-\mu|)^{2}||\chi_{k}u\chi_{j}v||_{2}^{2}\\
&\leq \sum_{k \leq 2\lambda, j \leq 2\mu} (\min(j,k))^{(d-2)}||(1+|\lambda-k|\chi_{k}u)||_{2}^{2}||(1+|\mu-j|)\chi_{j}v||_{2}^{2}\\
&\leq (\Lambda(d, \mu))^{2}(\sum_{n=0}^{\infty}||(1+|k-\lambda|)\chi_{k}u||_{2}^{2})(\sum_{n=0}^{\infty}||(1+|j-\mu|)\chi_{j}v||_{2}^{2})\\
& \leq (\Lambda(d, \mu))^{2}(\lambda^{-1}||(-\Delta_{g}-\lambda^{2})u||_{2} + ||u||_{2})^{2}(\mu^{-1}||(-\Delta_{g}-\mu^{2})v||_{2} + ||v||_{2})^{2},
\end{aligned}
\end{equation}
since the estimates for spectral clusters in \cite{burq3, burq4}:
\begin{equation*}
||\chi_{\lambda}u\chi_{\mu}v||_{2} \lesssim \Lambda(d, \min(\lambda, \mu))||u||_{2}||v||_{2}.    
\end{equation*}
This completes the proof of (\ref{qmr4}).

For the remaining case, we note that $q = \frac{2p}{p-2} > \frac{2d}{d-4}$ when $d \ge 6$ since $p \leq \frac{2d}{d-2}$. Then according to Theorem 1.3 in \cite{sogge3} corresponding to $V \equiv 0$, we have the estimates of $||v||_{q}$ in the right side of (\ref{vhu4}).
\begin{equation}\label{vq}
 ||v||_{q}  \lesssim \mu^{\sigma^(q)} (\mu^{-1}||(-\Delta_{g}-\mu^{2})v||_{2} + ||v||_{2} + \mu^{-N+d/2-\sigma(q)}||(I-\Delta_{g})^{N/2}R_{\mu}v||_{2}),   
\end{equation}
assuming that $N > d/2$ with $R_{\lambda}$ being the projection operator for $P = \sqrt{-\Delta_{g}}$ corresponding to the internal $[2\lambda, \infty)$. 
Thus, when $d \ge 6$ we have
\begin{equation}\label{vhu5}
\begin{aligned}
||vH_{\lambda}u||_{2} \leq& ||H_{\lambda}u||_{p}||v||_{q}\\ 
\leq &\lambda^{d(\frac{1}{2}-\frac{1}{p})-1}(\lambda^{-1}||(-\Delta_{g}-\lambda^{2})u||_{2} + ||u||_{2})\\
&\mu^{\sigma(q)}(\mu^{-1}||(-\Delta_{g}-\mu^{2})v||_{2} + ||v||_{2} + \mu^{-N+d/2-\sigma(q)}||(I-\Delta_{g})^{N/2}R_{\mu}v||_{2})\\
< &\mu^{\frac{d-2}{2}}(\lambda^{-1}||(-\Delta_{g}-\lambda^{2})u||_{2} + ||u||_{2})\\
&(\mu^{-1}||(-\Delta_{g}-\mu^{2})v||_{2} + ||v||_{2} + \mu^{-N+d/2-\sigma(q)}||(I-\Delta_{g})^{N/2}R_{\mu}v||_{2}),
\end{aligned}
\end{equation}
This along with (\ref{411}) and (\ref{414}) finish the proof of the bilinear estimates of quasimodes when $d \ge 6$.

Furthermore, if $\lambda = \mu$, we just need assume that $\frac{d}{2}(\frac{1}{2}-\frac{1}{p}) \leq 1$, that is, $p \leq \frac{2d}{d-4}$ to make the second inequality of (\ref{Hu10}) hold. Then we have $q = \frac{2p}{p-2}$ due to Holder inequality and  $q< \frac{2d}{d-4}$ when $d =6, 7$.  Thus if $\lambda = \mu$ for $d = 6, 7$, we can get a better result than the result of (\ref{vhu5}) due to the $L^{p}$ estimates of Theorem 1.3 in \cite{sogge3} corresponding to $V \equiv 0$ as follows:
\begin{equation}\label{vhu0005}
\begin{aligned}
||uH_{\lambda}u||_{2} \leq
\lambda^{\frac{d-2}{2}}(\lambda^{-1}||(-\Delta_{g}-\lambda^{2})u||_{2} + ||u||_{2})^{2}.
\end{aligned}
\end{equation}
Then along with (\ref{411}) and (\ref{414}), we have 

\begin{equation}\label{qmr005}
||u||_{4}^{2} \lesssim \Lambda(d, \lambda)(\lambda^{-1}||(-\Delta_{g}-\lambda^{2})u||_{2} + ||u||_{2})^{2}.
\end{equation}
If $\lambda = \mu$ for $d \ge 8$, we just need assume that $\frac{d}{2}(\frac{1}{2}-\frac{1}{p}) \leq 1$, that is, $p \leq \frac{2d}{d-4}$ to make the second inequality of (\ref{Hu10}) hold. Then we have
$q = \frac{2p}{p-2}$ due to Holder inequality and $q> \frac{2d}{d-4}$ for $d \ge 8$. Thus the result of $\lambda = \mu$ is the same as the result of $\lambda > \mu$ for $d \ge 8$ combining (\ref{vhu5}), (\ref{411}),  (\ref{414}) and the $L^{p}$ estimates of Theorem 1.3 in \cite{sogge3} corresponding to $V \equiv 0$. 

This finishes the proof of Theorem \ref{be4}.

\end{proof}

\section{Proof of approximation in $H^{-1}$}\label{res}

Proof. Under the above hypothesis, we claim that 
\begin{equation}\label{2.4}
||\mathcal{R_{\nu}}(u_{i}u_{j})||_{H^{-1}} \leq C\lambda_{\nu}^{-1}\Lambda(d, \min(\lambda_{m}, \lambda_{n})) \approx \nu^{-1/d}\Lambda(d, \min(\lambda_{m}, \lambda_{n})),
\end{equation}
the above is argued as in \cite{sogge2}.\\
Since
\begin{align*}
\nu^{-1/d} \Lambda(d, \min(\lambda_{m}, \lambda_{n})) < \varepsilon \Leftrightarrow \nu > \varepsilon^{-d} \Omega(d, \min(\lambda_{m}, \lambda_{n})).
\end{align*}
where $\mu(d)$, as in (\ref{2.3}), we conclude that in order to obtain (\ref{2.2}), we just need to prove (\ref{2.4}). To prove this we use Theorem \ref{be2}, Theorem \ref{be3} and Theorem \ref{be4} when $2 \leq d \leq 5$ to get 
\begin{equation*}
\begin{aligned}
||\mathcal{R_{\nu}}(u_{i}u_{j})||_{H^{-1}}^{2} &\leq C\lambda_{\nu}^{-2}\sum_{k>\nu}|<u_{i}u_{j},e_{k}>|^{2}\\
& \leq \lambda_{\nu}^{-2}||u_{i}u_{j}||_{2}^{2}\\ &\leq \lambda_{\nu}^{-2} \Lambda(d, \min(\lambda_{i}, \lambda_{j}))^{2}(\lambda_{i}^{-1}||(\Delta_{g} +   \lambda_{i}^{2})u||_{L_{2}}+||u||_{L_{2}})^{2}\\
&(\lambda_{j}^{-1}||(\Delta_{g} + \lambda_{j}^{2})v||_{L_{2}}+||v||_{L_{2}})^{2}\\
& \lesssim \lambda_{\nu}^{-2}\Lambda(d, \min(\lambda_{i}, \lambda_{j}))^{2}.
\end{aligned}
\end{equation*}
And we use Theorem \ref{be4} when $d \ge 6$ to get 
\begin{equation*}
\begin{aligned}
||\mathcal{R_{\nu}}(u_{i}u_{j})||_{H^{-1}}^{2} &\leq C\lambda_{\nu}^{-2}\sum_{k>\nu}|<u_{i}u_{j},e_{k}>|^{2}\\
& \leq \lambda_{\nu}^{-2}||u_{i}u_{j}||_{2}^{2}\\ &\leq \lambda_{\nu}^{-2} \Lambda(d, \min(\lambda_{i}, \lambda_{j}))^{2}(\lambda_{i}^{-1}||(\Delta_{g} +   \lambda_{i}^{2})u||_{L_{2}}+||u||_{L_{2}})^{2}\\
&(\lambda_{j}^{-1}||(\Delta_{g} + \lambda_{j}^{2})v||_{L_{2}}+||v||_{L_{2}} + \mu^{-N+d/2-\sigma(q)}||(I-\Delta_{g})^{N/2}R_{\mu}v||_{2})^{2}\\
& \lesssim \lambda_{\nu}^{-2}\Lambda(d, \min(\lambda_{i}, \lambda_{j}))^{2},
\end{aligned}
\end{equation*}
which are desired.

\section{Approximation of products for quasimodes}
The following result says that, if, as above, $i,j \leq n$ then the product of quasimodes $u_{i}u_{j}$ can be well approximated by elements of $B_{\nu}$ if $\nu$ is not much larger than $n$. In this section, we improve results in Lu-Steinerberger (\cite{lu}), since we can handle quasimodes.
\begin{theorem}\label{fix}
Fix $(M,g)$ as above. Then there is a $\sigma = \sigma_{d}$ so that if $\kappa = 1/d$ there is a uniform constant $C_{\kappa}$ such that if $i,j \leq n <\nu$ and $n \in \mathbb{N}$ we have
are desired.
\begin{equation}\label{Ru}
||\mathcal{R}_{\nu}(u_{i}u_{j})||_{L^{2}(M)} \leq Cn^{\sigma}(n/\nu)^{\kappa} 
\end{equation}
\end{theorem}

\begin{lemma}\label{for}
For $\sigma \in \mathbb{R}$, let $||f||_{H^{\sigma}(M)} = ||(I - \Delta_{g})^{\sigma/2}f||_{L^{2}(M)}$ denote the norm for the Sobolev space of the order $\sigma$ on $M$. If $n \in \mathbb{N}$ and if $1 \leq i \leq j \leq n$, then 
\begin{equation}\label{H}
||u_{i}u_{j}||_{H^{1}(M)} \leq C_{\mu,M}\lambda_{n}^{1+2\sigma(4)},
\end{equation}
if, for $p \ge 2$ we set
\begin{equation}\label{sigma}
\sigma(p) = \max\{d(\frac{1}{2} - \frac{1}{p}) - \frac{1}{2}, d(\frac{d-1}{2}(\frac{1}{2} - \frac{1}{p})\}.
\end{equation}
\end{lemma}

These bounds arise naturally from the estimates established by Koch-Tataru-Zworski \cite{ktz} saying that if $p \ge 2$ then for $j \ge 1$
 we have
\begin{equation}\label{Lp}
||u_{j}||_{L^{p}(M)} \lesssim \lambda_{j}^{\sigma(p)},
\end{equation}
with $\sigma(p)$ being as in \ref{sigma}.

\section{Proofs}
\subsection{Proof of Theorem \ref{fix}.}
\begin{proof}[Proof.]
To prove the $L^{2}$-estimate (\ref{Ru}) we note that 
\begin{equation}\label{Rnu}
||\mathcal{R}_{\nu}h||_{L^{2}}^{2} = \sum_{k>\nu}|<h,e_{k}>|^{2} \leq \lambda_{\nu}^{-2}\sum_{k>\nu}\lambda_{k}^{2}|<h,e_{k}>|^{2} \leq \lambda_{\nu}^{-2}||h||_{H^{1}}^{2}.
\end{equation}
If we take $h=u_{i}u_{j}$ and use this along with (\ref{H}) we conclude that 
\begin{equation}\label{Ruu}
||\mathcal{R}_{\nu}(u_{i}u_{j})||_{L^{2}} \lesssim \lambda_{\nu}^{1}\lambda_{n}^{1+2\sigma(4)}.
\end{equation}
Since, by the Weyl formula, $n \approx \lambda_{n}^{d}$ and $\nu \approx \lambda_{\nu}^{d}$, this inequality implies that
\begin{equation}\label{Rnuuu}
||\mathcal{R}_{\nu}(u_{i}u_{j})||_{L^{2}(M)} \leq C(n/\nu)^{1/d}n^{\frac{2}{d}\sigma(4)}.
\end{equation}
This of course yields (\ref{Ru}) with $\sigma$ there being $\frac{2}{d}\sigma(4)$.
\end{proof}

\subsection{Proof of Lemma \ref{for}.}
\begin{proof}[Proof]
To prove (\ref{H}) we first recall some basic facts about Sobolev spaces on manifolds. See \cite{sogge4} for more details. First, if $1 = \sum_{j=1}^{N}\varphi_{j}$ is a fixed smooth partition of unity with 
\begin{equation}\label{supp}
supp \varphi_{j} \Subset \Omega_{j},
\end{equation}
where $\Omega_{j} \subset M$ is a coordinate patch, we have fixed $\mu = 1$
\begin{equation}\label{fH}
||f||_{H^{1}(M)} \approx \sum_{j=1}^{N}\sum_{|\alpha| \leq 1}||\partial^{\alpha}(\varphi_{j}f)||_{L^{2}(\mathbb{R}^{n})}.
\end{equation}
Here, the $L^{2}$-norms are taken with respect to our local coordinates. $||u_{j_{1}}u_{j_{2}}||_{H^{1}(M)}$ is dominated by a finite sum of terms of the form
\begin{equation}\label{partial}
||\partial^{\alpha}(\varphi \cdot u_{j_{1}}u_{j_{2}})||_{L^{2}},
\end{equation}
where $\varphi = \varphi_{j}$ for some $j = 1, \dots, N$ and $|\alpha| \leq 1$. By Leibniz's rule, we can thus dominate the left side of (\ref{H}) by a finite sum of terms of the form 
\begin{equation}\label{Lk}
||L_{1}u_{j_{1}}L_{2}u_{j_{2}}||_{L^{2}(M)},
\end{equation}
where $L_{k} : C^{\infty}(M) \rightarrow C^{\infty}(M)$ are differential operators with smooth coefficients of order $m_{k}$ with 
\begin{equation}\label{m}
m_{1} + m_{2} \leq 1.
\end{equation}
As a result, by Holder's inequality, $||u_{j_{1}}u_{j_{2}}||_{H^{1}(M)}$ is majored by a finite sum of terms of the form 
\begin{equation}\label{productL}
\prod_{k=1}^{2}||L_{k}u_{j_{k}}||_{L^{4}(M)},
\end{equation}
where the $L_{k}$ are as above. Since $L_{k}$ is a differential operator of order $m_{k}$, for any $1<p<\infty$, we can obtain the following inequality based on the results in \cite{sogge1} and \cite{sogge3}.
\begin{equation}\label{L4}
\begin{aligned}
||L_{k}u_{j_{k}}||_{L^{4}(M)} 
&\lesssim \lambda_{n}^{m_{k} + \sigma(4)} (\lambda^{-1}||(\Delta_{g} + \lambda^{2})u_{j_{k}}||_{L_{2}}+||u_{j_{k}}||_{L_{2}}).
\end{aligned}
\end{equation}
By (\ref{m}) and (\ref{productL}), we obtain (\ref{H}) from this, which finishes the proof of Lemma \ref{for}.
\end {proof}
\subsection*{Acknowledgements}
\

\noindent

The research was carried out while the author was visiting Johns Hopkins University supervised by Professor C.D. Sogge. And the author would like to express her deep gratitude to Professor C.D. Sogge, for bringing this research topic to her attention, and also for the valuable guidance, helpful suggestions and comments he provided.

\bibliographystyle{plain}

\end{document}